\documentclass[a4paper,oneside,reqno]{amsart}
\usepackage[top=30mm,bottom=30mm,left=30mm,right=30mm]{geometry}
\usepackage{mathtools,amssymb}
\usepackage[foot]{amsaddr}
\usepackage[shortlabels]{enumitem}
\usepackage{comment}

\usepackage[
 bookmarksnumbered=true, colorlinks=true, allcolors=blue 
 ]{hyperref}

\usepackage[
 backend=biber, style=alphabetic, sorting=nyvt 
 ]{biblatex}

\usepackage[all]{xy}

\DeclareMathOperator{\Aut}{Aut}
\DeclareMathOperator{\Inn}{Inn}
\DeclareMathOperator{\Out}{Out}
\DeclareMathOperator{\Isom}{Isom}
\DeclareMathOperator{\Spec}{Spec}
\DeclareMathOperator{\Gal}{Gal}
\DeclareMathOperator{\GL}{GL}
\DeclareMathOperator{\ord}{ord}
\newcommand{\et}{\mathrm{\acute{e}t}}
\newcommand{\ab}{\mathrm{ab}}
\newcommand{\AbsGalGrp}[1]{G_{#1}}
\newcommand{\GeneSubgrp}[1]{\left\langle #1\right\rangle}
\newcommand{\GeneSubgrpTop}[1]{\overline{\GeneSubgrp{#1}}}
\newcommand{\DerivSubgrpTop}[1]{\overline{\left[#1,#1\right]}}
\newcommand{\DerivSeries}[2]{#1^{[#2]}}
\newcommand{\FinStepSolvQuo}[2]{#1^{#2}}
\newcommand{\FinStepSolvQuoGeom}[2]{#1^{(#2)}}
\newcommand{\Abelian}[1]{#1^{\ab}}
\newcommand{\tame}{\mathrm{tame}}
\newcommand{\EtFundGrpWithPt}[2]{\pi_{1}^{\et}(#1,#2)}
\newcommand{\EtFundGrp}[1]{\pi_{1}^{\et}(#1)}
\newcommand{\EtFundGrpTame}[1]{\pi_{1}^{\tame}(#1)}
\newcommand{\EtFundGrpTameWithPt}[2]{\pi_{1}^{\tame}(#1,#2)}
\newcommand{\EtFundGrpGeom}[1]{\Delta_{#1}}
\newcommand{\EtFundGrpPi}[1]{\Pi_{#1}}
\newcommand{\Centralizer}[2]{\mathrm{C}_{#1}\left(#2\right)}
\newcommand{\CenterSubgrp}[1]{\mathrm{Z}\left(#1\right)}
\newcommand{\UnivCov}[1]{\widetilde{#1}}

\usepackage{amsthm}

\numberwithin{equation}{section}

\theoremstyle{plain}
	\newtheorem{theorem}{Theorem}[section]
	\newtheorem*{theorem*}{Theorem}
	\newtheorem{proposition}[theorem]{Proposition}
	\newtheorem{lemma}[theorem]{Lemma}
	\newtheorem{corollary}[theorem]{Corollary}
	
	\newtheorem{Itheorem}{Theorem}
	
	\newtheorem{Ilemma}{Lemma}
	
	\newtheorem{Icorollary}{Corollary}

\theoremstyle{definition}
	\newtheorem{definition}[theorem]{Definition}
	\newtheorem{remark}[theorem]{Remark}

\title[Center-freeness of finite-step solvable profinite groups arising from anabelian geometry]{Center-freeness of finite-step solvable groups arising from anabelian geometry}

\date{Version of \today}

\author[N.~Yamaguchi]{Naganori~Yamaguchi}
\address{Tokyo Denki University, 5 Senju-Asahi-cho, Adachi-ku, Tokyo 120-8551, Japan}
\email{n.yamaguchi@mail.dendai.ac.jp}



\thanks{
This work was supported by JSPS KAKENHI Grant Numbers 23KJ0881.}

\begin{document}

\begin{abstract}
	Anabelian geometry suggests that, for suitably geometric objects, their \texorpdfstring{\'etale}{etale} fundamental groups determine the geometric objects up to isomorphism.
	From a group-theoretic viewpoint, this philosophy requires rigidity properties, which often follow from the center-freeness of the associated \texorpdfstring{\'etale}{etale} fundamental groups.
	In fact, some profinite groups arising from anabelian geometry are center-free.
	For any integer \texorpdfstring{$m\geq 2$}{m>=2}, we investigate how such center-freeness behaves under passage to the maximal \texorpdfstring{$m$}{m}-step solvable quotients.
	In particular, we show that the maximal \texorpdfstring{$m$}{m}-step solvable quotients of the \texorpdfstring{\'etale}{etale} and tame fundamental groups of a hyperbolic curve over a separably closed field are torsion-free and center-free.
	Furthermore, we show that this implies the rigidity property of the \texorpdfstring{$m$}{m}-step solvable Grothendieck conjecture.
\end{abstract}

\maketitle

\tableofcontents

\section*{Introduction}
Let $G$ be a profinite group.
We define the \textit{topological derived series} of $G$ by setting
\begin{equation*}
	\DerivSeries{G}{0}
	\coloneq G,
	\qquad
	\DerivSeries{G}{m}
	\coloneq \DerivSubgrpTop{\DerivSeries{G}{m-1}}
	\qquad (m\geq 1).
\end{equation*}
We set $\FinStepSolvQuo{G}{m} \coloneq G/\DerivSeries{G}{m}$ and call it the \textit{maximal $m$-step solvable quotient} of $G$.
Consider the following property:
\medskip
\begin{quote}
	For any $m\in\mathbb{Z}_{\geq 2}$, the quotient $\FinStepSolvQuo{G}{m}$ is center-free.
\end{quote}
\medskip
Known examples of profinite groups that satisfy the property include:
\medskip
\begin{quote}
	\textbf{Free pro-$\Sigma$ groups:}
	Free pro-$\Sigma$ groups are center-free.
	Moreover, the maximal $2$-step solvable quotients of free pro-$\Sigma$ groups are also center-free (see, for instance,~\cite[Section~4]{chen2017arithmeticmonodromyactionsprometabelian}).
	We can generalize this result from the case $m=2$ to all $m\in \mathbb{Z}_{\geq 2}$ immediately.
\end{quote}
\medskip
\begin{quote}
	\textbf{Absolute Galois groups:}
	The absolute Galois groups of number fields and of $p$-adic local fields are center-free.
	Moreover, for any $m\in\mathbb{Z}_{\geq 2}$, their maximal $m$-step solvable quotients are also center-free (see~\cite[Proposition~1.1(ix) and Corollary~1.7]{MR4460166}).
	This is closely related to the $m$-step solvable analogue of the Neukirch--Uchida theorem; see~\cite{MR4460166} for details.
\end{quote}
\medskip

If $G$ is metabelian and center-free, then for any $m\in\mathbb{Z}_{\geq 2}$ the natural projection $G\to \FinStepSolvQuo{G}{m}$ is an isomorphism, and hence $\FinStepSolvQuo{G}{m}$ is center-free.
In general, however, even if $G$ is center-free, the quotient $\FinStepSolvQuo{G}{m}$ need not be center-free.
In fact, we can easily construct a counterexample as follows:

\medskip
\begin{quote}
	Let $D_{8}=\langle r,s \mid r^{4}=1,\ s^{2}=1,\ srs^{-1}=r^{-1}\rangle$ be the dihedral group of order $8$, and define $\mathrm{pr}_{G}:D_{8}\to \GL_{2}(\mathbb{F}_{3})$ by
	\begin{equation*}
		r\longmapsto
		\begin{pmatrix}
			0 & -1 \\
			1 & 0
		\end{pmatrix},
		\qquad
		s\longmapsto
		\begin{pmatrix}
			1 & 0  \\
			0 & -1
		\end{pmatrix}.
	\end{equation*}
	Then $G \coloneq (C_{3}\times C_{3})\rtimes_{\mathrm{pr}_{G}} D_{8}$ is center-free; however, the quotient $\FinStepSolvQuo{G}{2}\cong D_{8}$ is not center-free.
\end{quote}
\medskip

In this paper, we give a new example of a profinite group that satisfies the property.
Let $k$ be a field (of arbitrary characteristic) with separable closure $\overline{k}$, and let $X$ be a smooth curve over $k$.
Note that we always assume that smooth curves are geometrically connected.
Let $\Sigma$ denote a non-empty set of prime numbers.
We write
\begin{equation*}
	\EtFundGrpWithPt{X}{\ast}\qquad (\text{resp. }\EtFundGrpTameWithPt{X}{\ast})
\end{equation*}
for the \textit{\'etale fundamental group} (resp. \textit{tame fundamental group}) of $X$, where $\ast:\Spec(\Omega)\to X$ denotes a geometric point of $X$ and $\Omega$ denotes an algebraically closed field.
The fundamental group depends on the choice of base point only up to inner automorphisms, and therefore we omit the choice of base point below.

The first main theorem of this paper is the following:

\begin{Itheorem}[Theorem~\ref{mainthm_center_free_fund}]\label{ItheoremA}
	Assume that $X$ is hyperbolic, that $k=\overline{k}$, and that $\Sigma$ contains a prime number different from the characteristic of $k$.
	Then, for any $m\in\mathbb{Z}_{\geq 2}$, the maximal $m$-step solvable quotients of $\EtFundGrp{X}^{\Sigma}$ and $\EtFundGrpTame{X}^{\Sigma}$ are both torsion-free and center-free.
\end{Itheorem}

\begin{Icorollary}[Corollary~\ref{center_free_surface}]\label{IcorollaryA}
	For any $m\in\mathbb{Z}_{\geq 2}$, the maximal $m$-step solvable quotient of a pro-$\Sigma$ surface group of genus at least $2$ is torsion-free and center-free.
\end{Icorollary}

We say that a profinite group $G$ is \textit{slim} if the centralizer $\Centralizer{G}{H}$ of each open subgroup $H\subset G$ in $G$ is trivial (see~\cite[Definition 0.1]{MR2059759}).
Since slimness is stronger than center-freeness, it is natural to ask whether the center-freeness statement in Theorem~\ref{ItheoremA} can be strengthened to slimness.
At the time of writing, the author does not know whether these groups are slim in general (see Proposition~\ref{climproposition} for partial results toward slimness).
To the best of the author’s knowledge, slimness for the $m$-step solvable quotients is currently known only for free pro-$\Sigma$ groups, as proved in~\cite[Section~1.1]{MR4578639}.
However, the argument of~\cite[Proposition~1.1.1]{MR4578639} contains an error and does not go through as written.
In Proposition~\ref{freeprocenter}, we provide a corrected proof of~\cite[Proposition~1.1.1]{MR4578639}, and in Section~\ref{free_section} we give a proof of the slimness of the $m$-step solvable quotients of free pro-$\Sigma$ groups as follows:

\begin{Itheorem}[Theorem~\ref{thm_center-free_freegroup} and Corollary~\ref{cor_slim_free}]\label{IthmB}
	Let $\mathcal{F}$ be a (possibly infinitely generated) free pro-$\Sigma$ group with a free generating set $\mathcal{X}$.
	Let $m\in\mathbb{Z}_{\geq 2}$.
	Then, for any nonzero integer $n\in\mathbb{Z}$ and any $x\in \mathcal{X}$, we have
	\begin{equation*}
		\Centralizer{\FinStepSolvQuo{\mathcal{F}}{m}}{x^n}
		=\GeneSubgrpTop{x}.
	\end{equation*}
	In particular, the quotient $\FinStepSolvQuo{\mathcal{F}}{m}$ is slim if $\mathcal{F}\not\cong \mathbb{Z}_{\Sigma}$.
\end{Itheorem}

\medskip

Next, we explain an application of Theorem~\ref{ItheoremA} to the $m$-step solvable analogue of the Grothendieck conjecture.
In the rest of the introduction, we focus only on the case where the field $k$ is a sub-$p$-adic field for some prime number $p$ (i.e., a field that embeds as a subfield of a finitely generated extension of $\mathbb{Q}_{p}$).
In particular, the field $k$ has characteristic $0$.
For simplicity, we write
\begin{equation*}
	\EtFundGrpGeom{X}
	\coloneq \EtFundGrp{X_{\overline{k}}}^{\Sigma},
	\qquad\text{and}\qquad
	\FinStepSolvQuoGeom{\EtFundGrpPi{X}}{m}
	\coloneq \EtFundGrp{X}/\ker (\EtFundGrp{X_{\overline{k}}}
	\to\FinStepSolvQuo{\EtFundGrpGeom{X}}{m}).
\end{equation*}
By construction, we have the following exact sequence:
\begin{equation*}
	1
	\to\FinStepSolvQuo{\EtFundGrpGeom{X}}{m}
	\to\FinStepSolvQuoGeom{\EtFundGrpPi{X}}{m}
	\to\AbsGalGrp{k}
	\to 1.
\end{equation*}
Here $\AbsGalGrp{k}$ denotes the absolute Galois group of $k$.

The original conjecture of A.~Grothendieck was first proposed in his letter to G.~Faltings~\cite{MR1483108} and was proved by S.~Mochizuki in~\cite{MR1720187}.
Moreover, in~\cite[Theorem~18.1]{MR1720187}, S.~Mochizuki proved the following ``existence'' statement for an $m$-step solvable analogue of the Grothendieck conjecture for hyperbolic curves over a sub-$p$-adic field $k$:
\medskip
\begin{quote}\textit{
		Assume $\Sigma=\{p\}$.
		Let $m\in\mathbb{Z}_{\geq 2}$.
		Let $X_{1}$ and $X_{2}$ be smooth curves over a sub-$p$-adic field $k$.
		Assume that at least one of $X_{1}$ and $X_{2}$ is hyperbolic.
		Then, for any $\AbsGalGrp{k}$-isomorphism
		\begin{equation*}
			\theta: \FinStepSolvQuoGeom{\EtFundGrpPi{X_{1}}}{m+3}
			\to
			\FinStepSolvQuoGeom{\EtFundGrpPi{X_{2}}}{m+3},
		\end{equation*}
		there exists a $k$-isomorphism $\mathrm{pr}_{G}: X_{1}\to X_{2}$ such that the $\AbsGalGrp{k}$-isomorphism $\FinStepSolvQuoGeom{\EtFundGrpPi{X_{1}}}{m}\to\FinStepSolvQuoGeom{\EtFundGrpPi{X_{2}}}{m}$ induced by $\mathrm{pr}_{G}$ (up to composition with an inner automorphism coming from $\FinStepSolvQuo{\EtFundGrpGeom{X_{2}}}{m}$) coincides with the isomorphism induced by $\theta$.}
\end{quote}
\medskip
With a little additional argument, this theorem can be reformulated as the surjectivity of the following natural map:
\medskip
\begin{quote}\textit{
		We keep the notation and assumptions as above.
		Then the natural map
		\begin{equation}\label{intro:m-step-morphism-mochi}
			\Isom_{\overline{k}/k}\bigl(\UnivCov{X_{1}}^{m}/X_{1}, \UnivCov{X_{2}}^{m}/X_{2}\bigr)
			\to
			\Isom_{\AbsGalGrp{k}}^{(m+3)}\bigl(\FinStepSolvQuoGeom{\EtFundGrpPi{X_{1}}}{m}, \FinStepSolvQuoGeom{\EtFundGrpPi{X_{2}}}{m}\bigr)
		\end{equation}
		is surjective, where $\UnivCov{X_{i}}^{m}\to X_{i}$ is the maximal geometrically $m$-step solvable pro-$\Sigma$ Galois covering of $X_{i}$, and the right-hand set is the image of the natural map
		\begin{equation*}
			\Isom_{\AbsGalGrp{k}}\bigl(\FinStepSolvQuoGeom{\EtFundGrpPi{X_{1}}}{m+3}, \FinStepSolvQuoGeom{\EtFundGrpPi{X_{2}}}{m+3}\bigr)
			\to
			\Isom_{\AbsGalGrp{k}}\bigl(\FinStepSolvQuoGeom{\EtFundGrpPi{X_{1}}}{m}, \FinStepSolvQuoGeom{\EtFundGrpPi{X_{2}}}{m}\bigr).
		\end{equation*}}
\end{quote}
\medskip
In this paper, we prove the injectivity statement as follows:

\begin{Itheorem}[Theorem~\ref{cor:relative_anabelian}]\label{IcorollaryB}
	We keep the notation and assumptions as above.
	Then the natural map~\eqref{intro:m-step-morphism-mochi} is bijective.
\end{Itheorem}

\section*{Notation and preliminaries in group theory}

For any profinite group $G$, we define the \textit{topological derived series} of $G$ by
$\DerivSeries{G}{0}\coloneq G$ and
\begin{equation*}
	\DerivSeries{G}{m} \coloneq \overline{[\DerivSeries{G}{m-1},\DerivSeries{G}{m-1}]}
	\qquad
	(m\geq 1),
\end{equation*}
where $\overline{[\DerivSeries{G}{m-1},\DerivSeries{G}{m-1}]}$ denotes the closed subgroup topologically generated by commutators of $\DerivSeries{G}{m-1}$.
For any $m\in\mathbb{Z}_{\geq 0}$, we set
\begin{equation*}
	\FinStepSolvQuo{G}{m} \coloneq G/\DerivSeries{G}{m},
\end{equation*}
and call it the \textit{maximal $m$-step solvable quotient} of $G$.
For simplicity, we write $\Abelian{G}$ for the abelianization of $G$.
With this notation, we have the following basic lemma:

\begin{Ilemma}\label{lemma:mstep}
	Let $f \colon G \to Q$ be a surjective morphism of profinite groups.
	Let $H \subset Q$ be an open subgroup and set $\tilde{H} \coloneq f^{-1}(H) \subset G$.
	Fix an integer $n\geq 0$.
	If $\ker(f) \subset \DerivSeries{\tilde{H}}{n}$, then the natural morphism $\FinStepSolvQuo{\tilde{H}}{n} \to \FinStepSolvQuo{H}{n}$ induced by $f$ is an isomorphism of profinite groups.
\end{Ilemma}

\begin{proof}
	Since profinite groups are compact Hausdorff, the image of a morphism (i.e., continuous homomorphism) is compact, hence closed.
	In particular, the morphism $f$ sends closed subgroups to closed subgroups.
	Hence we have $f(\DerivSeries{\tilde{H}}{n}) \subseteq \DerivSeries{H}{n}$.
	Since $f|_{\tilde{H}}\colon \tilde{H} \to H$ is surjective, the restriction $f|_{\DerivSeries{\tilde{H}}{n}}\colon \DerivSeries{\tilde{H}}{n} \to \DerivSeries{H}{n}$ is also surjective.
	Consider the commutative diagram with exact rows:
	\begin{equation*}
		\vcenter{\xymatrix@R=18pt@C=55pt{
		1 \ar[r] & \DerivSeries{\tilde{H}}{n} \ar[r] \ar@{->>}[d] & \tilde{H} \ar[r] \ar@{->>}[d] &
		\FinStepSolvQuo{\tilde{H}}{n} \ar[r] \ar[d] & 1 \\
		1 \ar[r] & \DerivSeries{H}{n} \ar[r] & H \ar[r] &
		\FinStepSolvQuo{H}{n} \ar[r] & 1 .
		}}
	\end{equation*}
	The kernel of the middle vertical morphism $\tilde{H} \to H$ is $\ker(f)$.
	By assumption $\ker(f)\subseteq \DerivSeries{\tilde{H}}{n}$, the kernel of the left-hand vertical morphism
	$\DerivSeries{\tilde{H}}{n} \to \DerivSeries{H}{n}$ is also $\ker(f)$.
	By the diagram chase, the right-hand vertical morphism is an isomorphism.
\end{proof}

We will frequently apply Lemma~\ref{lemma:mstep} in the setting that $Q \coloneq \FinStepSolvQuo{G}{m+n}$ and $H$ contains $\DerivSeries{(\FinStepSolvQuo{G}{m+n})}{m}$.
In this case, Lemma~\ref{lemma:mstep} shows that the natural surjection $\FinStepSolvQuo{\tilde{H}}{n} \twoheadrightarrow \FinStepSolvQuo{H}{n}$ is an isomorphism.
This observation is recorded in~\cite[Lemma~1.1]{MR4745885}.

\section{Centralizers in free \texorpdfstring{$m$}{m}-step solvable groups}\label{free_section}
In this section, we compute explicitly the centralizer of a free generator in a free $m$-step solvable pro-$\Sigma$ group.
A result of this form is stated in~\cite[Section~1.1]{MR4578639}; however, the proof of~\cite[Proposition~1.1.1]{MR4578639} contains an error and does not work as written.
In Proposition~\ref{freeprocenter} below, we provide a corrected argument.
Throughout this section, let $\Sigma$ be a non-empty set of prime numbers.
Moreover, for a profinite group $G$ and a subset $S\subset G$, we define
\begin{equation*}
	\Centralizer{G}{S}\coloneq \{g\in G \mid \forall s\in S,\, gs=sg\}
\end{equation*}
and call it the \textit{centralizer} of $S$ in $G$.
(Note that this group is already closed in $G$, and hence profinite.)
When $S=\{x\}$, we write $\Centralizer{G}{x}$ instead of $\Centralizer{G}{\{x\}}$ for simplicity.

\subsection{Pro-\texorpdfstring{$\Sigma$}{Sigma} Fox calculus and the Blanchfield--Lyndon sequence}
\subsubsection{}
We recall the pro-$\Sigma$ Fox calculus and the Blanchfield--Lyndon sequence.
For a pro-$\Sigma$ group $G$, we define its completed group ring by
\begin{equation*}
	\mathbb{Z}_{\Sigma}[[G]]
	\coloneq \varprojlim_{H,\ n}(\mathbb{Z}/n\mathbb{Z})[G/H],
\end{equation*}
where $H$ and $n$ run over all open normal subgroups of $G$ and all positive integers whose prime factors lie in $\Sigma$, respectively.
In~\cite{MR53938}, R.~H.~Fox developed the (discrete) free differential calculus.
Later, Y.~Ihara~\cite{MR1708605} established a pro-$\Sigma$ analogue for a finitely generated free pro-$\Sigma$ group $\mathcal{F}$ with free generating set $X=\{x_i\}_{1\leq i\leq r}$.
For any $i$, a continuous $\mathbb{Z}_{\Sigma}$-linear map
\begin{equation*}
	\partial_i:\mathbb{Z}_{\Sigma}[[\mathcal{F}]]\to \mathbb{Z}_{\Sigma}[[\mathcal{F}]]
\end{equation*}
satisfying the following properties is called the \textit{free differential} with respect to $x_i$:
\begin{enumerate}[(i)]
	\item
	      $\partial_i(1)=0$, where $1$ is the unit of $\mathbb{Z}_{\Sigma}[[\mathcal{F}]]$;
	\item
	      $\partial_i(x_j)=\delta_{i,j}$;
	\item
	      for any $\lambda,\tilde{\lambda}\in\mathbb{Z}_{\Sigma}[[\mathcal{F}]]$, we have
	      \begin{equation*}
		      \partial_i(\lambda\tilde{\lambda})
		      =\partial_i(\lambda)\,s(\tilde{\lambda})+\lambda\,\partial_i(\tilde{\lambda}),
	      \end{equation*}
	      where $s$ is the augmentation morphism $\mathbb{Z}_{\Sigma}[[\mathcal{F}]]\to \mathbb{Z}_{\Sigma}$.
\end{enumerate}
For each $i$, such a free differential is uniquely determined; see~\cite[Appendix]{MR1708605}.
Moreover, every $\lambda\in\mathbb{Z}_{\Sigma}[[\mathcal{F}]]$ admits an expansion
\begin{equation*}
	\lambda = s(\lambda)\cdot 1 + \sum_{i=1}^r \partial_{i}(\lambda)(x_i-1),
\end{equation*}
and this expansion is unique (see~\cite[Theorem~A-1]{MR1708605}).

\subsubsection{}
Let $\mathcal{N}$ be a closed normal subgroup of $\mathcal{F}$.
The conjugation action of $\mathcal{F}/\mathcal{N}$ on $\Abelian{\mathcal{N}}$ extends continuously to an action of $\mathbb{Z}_{\Sigma}[[\mathcal{F}/\mathcal{N}]]$.
We regard $\Abelian{\mathcal{N}}$ as a $\mathbb{Z}_{\Sigma}[[\mathcal{F}/\mathcal{N}]]$-module by this action.
Let $\pi:\mathbb{Z}_{\Sigma}[[\mathcal{F}]]\to \mathbb{Z}_{\Sigma}[[\mathcal{F}/\mathcal{N}]]$ be the natural projection.
For each $i$, define
\begin{equation*}
	\tilde{\iota}:\mathcal{N}\to \mathbb{Z}_{\Sigma}[[\mathcal{F}/\mathcal{N}]]^{\oplus r};\qquad
	\tilde{\iota}(n)\coloneq \bigl(\pi\circ \partial_i(n)\bigr)_{1\leq i\leq r}.
\end{equation*}
Since $\pi(n)=1$ for each $n\in\mathcal{N}$, we have $\tilde{\iota}(n_1n_2)=\tilde{\iota}(n_1)+\tilde{\iota}(n_2)$.
Therefore, the continuous map $\tilde{\iota}$ is a homomorphism and factors through $\Abelian{\mathcal{N}}$.
We write $\iota$ for the induced morphism
\begin{equation*}
	\Abelian{\mathcal{N}}\to
	\mathbb{Z}_{\Sigma}[[\mathcal{F}/\mathcal{N}]]^{\oplus r}.
\end{equation*}
Using the free differentials, Y.~Ihara proved the profinite Blanchfield--Lyndon sequence:

\begin{proposition}[The Blanchfield--Lyndon exact sequence; {see~\cite[Theorem~A-2]{MR1708605}}]\label{bLtheory}
	Let $\mathcal{F}$ be a free pro-$\Sigma$ group of finite rank $r$ with free generating set $X=\{x_i\}_{1\leq i\leq r}$, and let $\mathcal{N}$ be a closed normal subgroup of $\mathcal{F}$.
	Then the sequence
	\begin{equation*}
		0\to \Abelian{\mathcal{N}}\xrightarrow{\iota}
		\mathbb{Z}_{\Sigma}[[\mathcal{F}/\mathcal{N}]]^{\oplus r}
		\xrightarrow{f}
		\mathbb{Z}_{\Sigma}[[\mathcal{F}/\mathcal{N}]]
		\xrightarrow{s}
		\mathbb{Z}_{\Sigma}\to 0
	\end{equation*}
	of $\mathbb{Z}_{\Sigma}[[\mathcal{F}/\mathcal{N}]]$-modules is exact.
	Here, the morphism $f$ is given by
	\begin{equation*}
		f\bigl((\lambda_1,\cdots,\lambda_{r})\bigr)=\sum_{i=1}^r \lambda_i(\pi(x_i)-1).
	\end{equation*}
\end{proposition}

The Blanchfield--Lyndon exact sequence admits a generalization to arbitrary profinite groups, known as the \textit{complete Crowell exact sequence}; see~\cite[Section~10.4]{MR4770091} for details.

\subsection{A computation of a centralizer in a free pro-\texorpdfstring{$\Sigma$}{Sigma} product}
\subsubsection{}
A slightly different version of the following proposition first appeared in~\cite[Lemma~2.1.2]{MR1298541}, where it was used to prove the center-freeness of free discrete groups.
We generalize it to our setting as follows:

\begin{lemma}\label{groupcl}
	Let $G$ be a finite group, let $x\in G$, and put
	$H\coloneq\GeneSubgrp{x}$.
	Let $s$ be the order of $x$ in $G$, and let $N$ be a positive integer such that $N\mid s$.
	Let $n$ be a positive integer whose image in $\mathbb{Z}/N\mathbb{Z}$ is nonzero.
	Let $G\times H$ act on the free $\mathbb{Z}/N\mathbb{Z}$-module
	$(\mathbb{Z}/N\mathbb{Z})[G/H]$ through the natural left action of the first factor $G$ on the set of left cosets $G/H$; the second factor $H$ acts trivially.
	Define
	\begin{equation*}
		\widetilde G
		\coloneq
		(\mathbb{Z}/N\mathbb{Z})[G/H]\rtimes (G\times H).
	\end{equation*}
	Let $\tau\coloneq ([H],(x,x))\in\widetilde G$, and let
	$\operatorname{pr}_{G}:\widetilde G\to G$ be the natural projection.
	If $z\in\widetilde G$ satisfies
	$z\tau^{n}z^{-1}\in\GeneSubgrp{\tau}$, then
	$\operatorname{pr}_{G}(z)\in H$.
\end{lemma}

\begin{proof}
	Since $x\in H$, the element $x$ fixes the coset $H\in G/H$.
	Hence, for every $r\geq 0$, we have
	$\tau^{r}=(r[H],(x^{r},x^{r}))$.
	In particular, $\tau^{s}=1$, since $x^s=1$ and $s[H]=0$ in
	$(\mathbb{Z}/N\mathbb{Z})[G/H]$.
	Write $z=({\bf b},(a,h))$, where
	${\bf b}\in(\mathbb{Z}/N\mathbb{Z})[G/H]$, $a\in G$, and $h\in H$.
	Assume that $z\tau^{n}z^{-1}\in\GeneSubgrp{\tau}$.
	Since $\tau$ has finite order, there exists an integer $r\geq 0$ such that
	$z\tau^{n}z^{-1}=\tau^{r}$.  Equivalently, $z\tau^{n}=\tau^{r}z$.
	Using the multiplication rule, we compute
	\begin{equation*}
	\begin{aligned}
		z\tau^{n}
		&=&
		({\bf b}+a(n[H]),(ax^{n},hx^{n})),\\
		\tau^{r}z
		&=&
		(r[H]+x^{r}{\bf b},(x^{r}a,x^{r}h)).
	\end{aligned}
	\end{equation*}
	Therefore, comparing the second components in $G\times H$, we obtain
	$hx^{n}=x^{r}h$ in $H$.
	Since $H=\GeneSubgrp{x}$ is cyclic, the element $h$ commutes with $x$.
	Hence $hx^{n}=x^{r}h$ implies $x^{n}=x^{r}$.
	Thus, $r\equiv n\pmod{s}$, and since $N\mid s$, we have
	$r=n$ in $\mathbb{Z}/N\mathbb{Z}$.

	Next compare the first components of $z\tau^{n}=\tau^{r}z$.
	From the computations above, we obtain
	\begin{equation*}
		{\bf b}+a(n[H])=r[H]+x^{r}{\bf b}.
	\end{equation*}
	Now take the coefficient of the basis vector $[H]$ on both sides.
	Since $x^r\in H$, the element $x^r$ fixes the coset $H$, so the coefficient of $[H]$ in $x^r{\bf b}$ is the same as the coefficient of $[H]$ in ${\bf b}$.
	Therefore the coefficient of $[H]$ in $a(n[H])$ is equal to $r$, hence equal to $n$ in $\mathbb{Z}/N\mathbb{Z}$.
	On the other hand, $a[H]=[aH]$.  
	Thus, the coefficient of $[H]$ in
	$a(n[H])=n[aH]$ is $n$ if $aH=H$, and is $0$ if $aH\neq H$.
	Hence the hypothesis $n\neq 0$ in $\mathbb{Z}/N\mathbb{Z}$ gives $a\in H$.
\end{proof}

\begin{proposition}\label{freeprocenter}
	Let $\Omega=\mathcal{C}*P$ be the free pro-$\Sigma$ product (see~\cite[Proposition~9.1.2]{MR2599132})
	of a procyclic pro-$\Sigma$ group $\mathcal{C}$, topologically generated by an element $x$, and a pro-$\Sigma$ group $P$.
	Let $m\in\mathbb{Z}_{\geq 2}$.
	Then, for any $n\in\mathbb{Z}$ such that $x^{n}\neq 1$ in $\mathcal{C}$, we have
	\begin{equation}\label{freeprocenter:eq}
		\Centralizer{\FinStepSolvQuo{\Omega}{m}}{x^{n}}
		\subset
		\GeneSubgrpTop{x} \cdot \DerivSeries{(\FinStepSolvQuo{\Omega}{m})}{m-1}
	\end{equation}
	as a subgroup of $\FinStepSolvQuo{\Omega}{m}$, where $\GeneSubgrpTop{x}$ denotes the closed subgroup of $\FinStepSolvQuo{\Omega}{m}$ topologically generated by the image of $x$.
\end{proposition}

\begin{proof}
	Since $x^{-1}$ is also a topological generator of $\mathcal{C}$, we may assume that $n\geq 1$.
	To prove~\eqref{freeprocenter:eq}, it suffices to show that, for any surjective morphism 
	$\rho:\FinStepSolvQuo{\Omega}{m}\twoheadrightarrow G$ onto a finite group $G$ that factors through the natural projection
	$\FinStepSolvQuo{\Omega}{m}\twoheadrightarrow\FinStepSolvQuo{\Omega}{m-1}$, if we put
	$H\coloneq \GeneSubgrp{\rho(x)}$, then we have
	\begin{equation}\label{gragaerghafw}
		\rho\left(\Centralizer{\FinStepSolvQuo{\Omega}{m}}{x^{n}}\right)\subset H.
	\end{equation}
	Since $\Omega=\mathcal{C}*P$, we have $\Abelian{\Omega}\cong \Abelian{\mathcal{C}}\times \Abelian{P}\cong \mathcal{C}\times \Abelian{P}$.
	In particular, the composition of the natural morphisms
	$\mathcal{C}\to \Omega \to \FinStepSolvQuo{\Omega}{m-1}$ is injective.
	Hence the family of surjections $\rho$ such that $\rho(x^{n})\neq 1$ is cofinal.
	Therefore, we may assume that $\rho(x^{n})\neq 1$ in the above.
	To prove~\eqref{gragaerghafw}, it suffices to show the following condition:
	\begin{quote}\it
	there exists a finite pro-$\Sigma$ group $\tilde{G}$ and a factorization
	\begin{equation}\label{desiredfactrization}
		\vcenter{
\xymatrix{ 
	&\FinStepSolvQuo{\Omega}{m}\ar[d]^-{\rho}\ar[ld]_-{\psi}\\ \tilde{G}\ar[r]^{\mathrm{pr}_{G}}&G 
}
}
	\end{equation}	
	of $\rho$ such that, if we write $\tau\coloneq \psi(x)$, then 
	\begin{equation}\label{hthetharwga}
		\mathrm{pr}_{G}\bigl(\Centralizer{\tilde{G}}{\tau^{n}}\bigr)\subset H.
	\end{equation}
\end{quote}
	Indeed, condition~\eqref{gragaerghafw} follows from
	\[
		\rho\bigl(\Centralizer{\FinStepSolvQuo{\Omega}{m}}{x^{n}}\bigr)
		=(\mathrm{pr}_{G}\circ\psi)\bigl(\Centralizer{\FinStepSolvQuo{\Omega}{m}}{x^{n}}\bigr)
		\subset \mathrm{pr}_{G}\bigl(\Centralizer{\tilde{G}}{\tau^{n}}\bigr)
		\subset H.
	\]

	Let us construct the desired group $\tilde{G}$ and morphism
	$\psi:\FinStepSolvQuo{\Omega}{m}\to \tilde{G}$.
	Let $s$ be the order of $\rho(x)$ in $G$.
	Since $\rho(x^{n})\neq 1$, we have $s\nmid n$.
	Since $G$ is a finite $\Sigma$-group, there exists $\ell\in\Sigma$ such that
	$r\coloneq\ord_{\ell}(s)>\ord_{\ell}(n)$.
	Put $N\coloneq \ell^r$.
	In particular, $s=0$ and $n\neq0$ in $\mathbb{Z}/N\mathbb{Z}$.
	Let $G\times H$ act on the free $\mathbb{Z}/N\mathbb{Z}$-module
	$(\mathbb{Z}/N\mathbb{Z})[G/H]$ through the natural left action of the first factor $G$ on the set of left cosets $G/H$; the second factor $H$ acts trivially.
	Define 
	\[
	\tilde{G}\coloneq (\mathbb{Z}/N\mathbb{Z})[G/H]\rtimes (G\times H).
	\]
	Define $\tau\coloneq ([H],(\rho(x),\rho(x)))\in \tilde{G}$.
	Since $\rho(x)$ fixes the coset $H$, we have
	\[
		\tau^{s}=(s[H],(\rho(x)^{s},\rho(x)^{s}))=(0,(1,1)),
	\]
	which is the identity element of $\tilde{G}$.
	Hence we obtain a morphism $H\to \tilde{G}$ by sending $\rho(x)$ to $\tau$.
	Moreover, the composite
	$\mathcal{C}\to\Omega\to\FinStepSolvQuo{\Omega}{m}\xrightarrow{\rho}G$
	factors through $H$, because $\mathcal{C}$ is topologically generated by $x$.
	Thus we have the following commutative diagram:
	\[
	\xymatrix{
		\Omega \ar[r]& \FinStepSolvQuo{\Omega}{m}\ar[r]^-{\rho}& G\\
		\mathcal{C}\ar[u]\ar[r]& H\ar[ur]\ar[r]& \tilde{G} \ar[u]_-{\mathrm{pr}_{G}}
	}
	\]
	Moreover, we define $P\to \tilde{G}$ by
	$p\mapsto (0,(\rho(p),1))$, where $\rho(p)$ denotes the image of $p$ under
	$P\to\Omega\to\FinStepSolvQuo{\Omega}{m}\xrightarrow{\rho}G$.
	Since $\tilde{G}$ is a finite pro-$\Sigma$ group, the universal property of the free pro-$\Sigma$ product induces a morphism
	$\Omega\to \tilde{G}$.
	Since $G$ is an $(m-1)$-step solvable group, $H$ is abelian, and
	$(\mathbb{Z}/N\mathbb{Z})[G/H]$ is abelian, the group $\tilde{G}$ is an $m$-step solvable group.
	Hence the morphism $\Omega\to \tilde{G}$ induces a morphism
	$\psi : \FinStepSolvQuo{\Omega}{m}\to \tilde{G}$.
	By construction, $\psi(x)=\tau$ and $\rho=\mathrm{pr}_{G}\circ\psi$.
	The morphism $\psi$, the group $\tilde{G}$, and the projection $\mathrm{pr}_{G}$ satisfy the desired condition~\eqref{hthetharwga} by Lemma~\ref{groupcl}.
	This completes the proof.
\end{proof}

\subsection{Proof of the slimness of free \texorpdfstring{$m$}{m}-step solvable groups}

\subsubsection{}
Using the above ingredients, we compute explicitly the centralizer of a free generator in a (possibly infinitely generated) free $m$-step solvable pro-$\Sigma$ group and deduce the slimness of such profinite groups.

\begin{lemma}\label{lem:non-zero-div-free}
	Let $\mathcal{F}$ be a free pro-$\Sigma$ group of finite rank $r$ with free generating set $X$.
	For any nonzero integer $n\in\mathbb{Z}$ and any $x\in X$, the element $\underline{x}^n-1$ is a nonzero divisor in $\mathbb{Z}_{\Sigma}[[\Abelian{\mathcal{F}}]]$, where $\underline{x}$ is the image of $x$ in $\Abelian{\mathcal{F}}$.
\end{lemma}

\begin{proof}
	Denote by $\mathbb{Z}(\Sigma)_{\geq 1}$ the set of all positive integers whose prime factors lie in $\Sigma$.
	We may assume that $n\geq 1$ since $\underline{x}^{-n}-1=-\underline{x}^{-n}(\underline{x}^{\,n}-1)$ in $\mathbb{Z}_{\Sigma}[[\Abelian{\mathcal{F}}]]$.
	We show that if $y \in \mathbb{Z}_{\Sigma}[[\Abelian{\mathcal{F}}]]$ satisfies $(\underline{x}^n-1)y=0$, then $y=0$.

	Since $\Abelian{\mathcal{F}}$ is a free $\mathbb{Z}_{\Sigma}$-module of finite rank $r$, we may identify
	\begin{equation*}
		\Abelian{\mathcal{F}}\, \cong\, H \times \mathbb{Z}_{\Sigma},
	\end{equation*}
	where $H\cong \mathbb{Z}_{\Sigma}^{r-1}$ is the free abelian factor generated by the images of $X\setminus\{x\}$, and the factor $\mathbb{Z}_{\Sigma}$ corresponds to $\underline{x}$.
	Put
	\begin{equation*}
		A \coloneq \mathbb{Z}_{\Sigma}[[H]].
	\end{equation*}
	For each $N\in\mathbb{Z}(\Sigma)_{\geq 1}$, let $C_N\coloneq \langle \underline{x}_{N} \mid (\underline{x}_{N})^N=1\rangle\cong \mathbb{Z}/N\mathbb{Z}$.
	Then, by the definition of the completed group algebra and the above decomposition, we have
	\begin{equation*}
		\mathbb{Z}_{\Sigma}[[\Abelian{\mathcal{F}}]]\, \cong\, \varprojlim_{N\in\mathbb{Z}(\Sigma)_{\geq 1}} A[C_N].
	\end{equation*}
	Here, we may regard $\underline{x}$ as the projective limit of $(\underline{x}_{N})_{N}$.

	Write $y_N$ for the image of $y$ in $A[C_N]$.
	Since $\{1,\underline{x}_{N},\cdots,(\underline{x}_{N})^{N-1}\}$ is an $A$-basis of $A[C_N]$, there exists a unique $c_i^{(N)}\in A$ such that
	\begin{equation*}
		y_N=\sum_{i=0}^{N-1} c_i^{(N)} \underline{x}_{N}^{i}.
	\end{equation*}
	The equation $(\underline{x}^{n}-1)y=0$ implies $((\underline{x}_{N})^n-1)y_N=0$ for any $N$, and hence
	\begin{equation*}
		0=((\underline{x}_{N})^n-1)\left(\sum_{i=0}^{N-1} c_i^{(N)} \underline{x}_{N}^{i}\right)
		=\sum_{i=0}^{N-1} \bigl(c_{i-n}^{(N)}-c_i^{(N)}\bigr)\underline{x}_{N}^{i},
	\end{equation*}
	where the indices $i$ of $c_i^{(N)}$ are taken in $\mathbb{Z}/N\mathbb{Z}$.
	By $A$-linear independence of $\{\underline{x}_{N}^i\}$, we obtain $c_{i-n}^{(N)}=c_i^{(N)}$ for each $i\in\mathbb{Z}/N\mathbb{Z}$.
	In other words, the coefficients $c_i^{(N)}$ are constant on cosets of the subgroup $\langle n\rangle\subset \mathbb{Z}/N\mathbb{Z}$.

	Let $n=n_\Sigma\cdot n_{\Sigma'}$ be the unique decomposition such that $n_\Sigma\in\mathbb{Z}(\Sigma)_{\geq 1}$ and that $n_{\Sigma'}$ is coprime to all primes in $\Sigma$.
	Fix $M\in\mathbb{Z}(\Sigma)_{\geq 1}$ such that $n_\Sigma\mid M$, and let $k\in\mathbb{Z}(\Sigma)_{\geq 1}$ be arbitrary.
	As $n_{\Sigma'}$ and $kM$ are coprime to each other, we have $\langle n\rangle=\langle n_\Sigma\rangle\subset \mathbb{Z}/kM\mathbb{Z}$.
	Hence we may apply the above result with $N=kM$, which gives $c_{i-n_\Sigma}^{(kM)}=c_i^{(kM)}$ for each $i\in\mathbb{Z}/kM\mathbb{Z}$.
	Therefore, by $n_\Sigma\mid M$, we obtain
	\begin{equation}\label{eq1_lem:non-zero-div-free}
		c_{i}^{(kM)}=c_{i+M}^{(kM)}=\cdots=c_{i+(k-1)M}^{(kM)}
	\end{equation}
	for each $i\in\mathbb{Z}/kM\mathbb{Z}$.
	Let $\pi:A[C_{kM}]\to A[C_M],\ \underline{x}_{kM}\mapsto \underline{x}_M$, be the natural projection induced by $\mathbb{Z}/kM\mathbb{Z}\twoheadrightarrow \mathbb{Z}/M\mathbb{Z}$.
	By $\pi(\underline{x}_{kM})=\underline{x}_M$ and~\eqref{eq1_lem:non-zero-div-free}, we have
	\begin{equation*}
		y_M=\pi(y_{kM})
		=\sum_{i=0}^{kM-1} c_i^{(kM)}\,\pi\bigl(\underline{x}_{kM}^i\bigr)
		=\sum_{i=0}^{M-1}\Bigl(\sum_{j=0}^{k-1} c_{i+jM}^{(kM)}\Bigr)\underline{x}_M^{i}=\sum_{i=0}^{M-1}\bigl(k\cdot c_{i}^{(kM)}\bigr)\underline{x}_M^{i}.
	\end{equation*}
	Comparing this with $y_M=\sum_{i=0}^{M-1} c_i^{(M)}\underline{x}_M^{i}$, we obtain
	\begin{equation*}
		c_i^{(M)}=k\cdot c_{i}^{(kM)}\in kA
	\end{equation*}
	for each $i\in \mathbb{Z}/M\mathbb{Z}$.
	By running over all $k\in\mathbb{Z}(\Sigma)_{\geq 1}$ and using the fact $\bigcap_{k} kA=\{0\}$, we obtain $y_M=0$.
	Since the set $\{M\in\mathbb{Z}(\Sigma)_{\geq 1}\mid n_\Sigma\mid M\}$ is cofinal in $\mathbb{Z}(\Sigma)_{\geq 1}$, it follows that $y=0$.
	This completes the proof.
\end{proof}

\begin{theorem}\label{thm_center-free_freegroup}
	Let $\mathcal{F}$ be a (possibly infinitely generated) free pro-$\Sigma$ group of rank $r$ with free generating set $X$.
	Let $m\in\mathbb{Z}_{\geq 2}$.
	Then, for any nonzero integer $n\in\mathbb{Z}$ and any $x\in X$, we have
	\begin{equation*}
		\Centralizer{\FinStepSolvQuo{\mathcal{F}}{m}}{x^n} = \GeneSubgrpTop{x}.
	\end{equation*}
\end{theorem}

\begin{proof}
	If $r=1$, the assertion is clear.
	Hence we may assume $r\neq 1$.
	Fix $x\in X$.
	We divide the proof into three cases: $m=2$ with finite $r$; general $m$ with finite $r$; and the case $r=\infty$.

	First, we assume that $m=2$ and $r$ is finite.
	By Proposition~\ref{freeprocenter} and the fact that $\GeneSubgrpTop{x}\subset \Centralizer{\FinStepSolvQuo{\mathcal{F}}{2}}{x^{n}}$, we obtain
	$\Centralizer{\FinStepSolvQuo{\mathcal{F}}{2}}{x^{n}}=\GeneSubgrpTop{x}\cdot(\Centralizer{\FinStepSolvQuo{\mathcal{F}}{2}}{x^{n}}\cap\DerivSeries{(\FinStepSolvQuo{\mathcal{F}}{2})}{1})$.
	Therefore, it suffices to show that
	\begin{equation}\label{eq1:thm_center-free_freegroup}
		\Centralizer{\FinStepSolvQuo{\mathcal{F}}{2}}{x^{n}}\cap\DerivSeries{(\FinStepSolvQuo{\mathcal{F}}{2})}{1}=1.
	\end{equation}
	Applying Proposition~\ref{bLtheory} to the case $\mathcal{N}=\mathcal{F}^{[1]}$, we obtain an injective $\mathbb{Z}_{\Sigma}[[\Abelian{\mathcal{F}}]]$-linear morphism
	\begin{equation*}
		\iota:\ \DerivSeries{(\FinStepSolvQuo{\mathcal{F}}{2})}{1}\hookrightarrow \mathbb{Z}_{\Sigma}[[\Abelian{\mathcal{F}}]]^{\oplus r}.
	\end{equation*}
	Consider the conjugation action of $x^n$ on the abelian group $\DerivSeries{(\FinStepSolvQuo{\mathcal{F}}{2})}{1}$.
	By $\mathbb{Z}_{\Sigma}[[\Abelian{\mathcal{F}}]]$-linearity of $\iota$, we obtain
	\begin{eqnarray*}
		\Centralizer{\FinStepSolvQuo{\mathcal{F}}{2}}{x^{n}}\cap\DerivSeries{(\FinStepSolvQuo{\mathcal{F}}{2})}{1}&=&\{u\in \DerivSeries{(\FinStepSolvQuo{\mathcal{F}}{2})}{1}\mid x^{n} u x^{-n}=u\}\\
		&=&\ker\bigl((\underline{x}^{n}-1):\DerivSeries{(\FinStepSolvQuo{\mathcal{F}}{2})}{1}\to \DerivSeries{(\FinStepSolvQuo{\mathcal{F}}{2})}{1}\bigr)\\
		&\subset&
		\ker\bigl((\underline{x}^{n}-1):\mathbb{Z}_{\Sigma}[[\Abelian{\mathcal{F}}]]^{\oplus r}\to \mathbb{Z}_{\Sigma}[[\Abelian{\mathcal{F}}]]^{\oplus r}\bigr),
	\end{eqnarray*}
	where $\underline{x}$ is the image of $x$ in $\Abelian{\mathcal{F}}$.
	By Lemma~\ref{lem:non-zero-div-free}, the element $\underline{x}^{n}-1$ is a nonzero divisor in $\mathbb{Z}_{\Sigma}[[\Abelian{\mathcal{F}}]]$, and hence multiplication by $\underline{x}^{n}-1$ is injective.
	Therefore, the last kernel is trivial, and hence the equation~\eqref{eq1:thm_center-free_freegroup} follows.
	This proves
	\begin{equation*}
		\Centralizer{\FinStepSolvQuo{\mathcal{F}}{2}}{x^n}=\GeneSubgrpTop{x}
	\end{equation*}
	in the case where $r$ is finite.

	Next, assume that $r$ is finite and proceed by induction on $m\in \mathbb{Z}_{\geq 2}$.
	The case of $m=2$ is already proved.
	Assume that $m>2$ and that the assertion holds for $m-1$.
	As in the case $m=2$, by Proposition~\ref{freeprocenter}, it suffices to show that
	\begin{equation}\label{eq2:thm_center-free_freegroup}
		\Centralizer{\FinStepSolvQuo{\mathcal{F}}{m}}{x^{n}}\cap\DerivSeries{(\FinStepSolvQuo{\mathcal{F}}{m})}{m-1}=1.
	\end{equation}
	Let $g$ be an element of the left-hand side of~\eqref{eq2:thm_center-free_freegroup}.
	Let $H$ be an open normal subgroup of $\FinStepSolvQuo{\mathcal{F}}{m}$ that contains $\DerivSeries{(\FinStepSolvQuo{\mathcal{F}}{m})}{1}$.
	Since
	\begin{equation*}
		\bigcap_{H} \DerivSeries{H}{m-1}=\DerivSeries{(\DerivSeries{(\FinStepSolvQuo{\mathcal{F}}{m})}{1})}{m-1}=1,
	\end{equation*}
	it suffices to show that $\rho_{H}(g)=1$, i.e., the condition $g\in \DerivSeries{H}{m-1}$ holds, for each such $H$, where $\rho_{H}:H\twoheadrightarrow \FinStepSolvQuo{H}{m-1}$ is the natural surjection.
	The image of $x$ in the finite quotient $\FinStepSolvQuo{\mathcal{F}}{m}/H$ has finite order.
	Let $N$ denote this order.
	Since $g$ commutes with $x^{n}$, it also commutes with $x^{Nn}$, and therefore $\rho_H(g)$ commutes with $\rho_H(x^{Nn})$.
	By the Nielsen--Schreier theorem, the inverse image $\tilde{H}$ of $H$ in $\mathcal{F}$ is again a free pro-$\Sigma$ group, and we may choose a free generating set of $\tilde{H}$ that contains $x^N$.
	By Lemma~\ref{lemma:mstep}, we have $\FinStepSolvQuo{H}{m-1}\cong \FinStepSolvQuo{\tilde{H}}{m-1}$.
	Applying the induction hypothesis for $m-1$ to $\FinStepSolvQuo{\tilde{H}}{m-1}$ and the basis element $x^N\in\FinStepSolvQuo{\tilde{H}}{m-1}$, we obtain
	\begin{equation*}
		\Centralizer{\FinStepSolvQuo{H}{m-1}}{(x^N)^n}=\GeneSubgrpTop{x^N}.
	\end{equation*}
	On the other hand, we have $g\in (\FinStepSolvQuo{\mathcal{F}}{m})^{[m-1]}\subset (\FinStepSolvQuo{\mathcal{F}}{m})^{[2]}\subset H^{[1]}$ and hence $\rho_H(g)\in (\FinStepSolvQuo{H}{m-1})^{[1]}$.
	Note that $\GeneSubgrpTop{x^N}$ embeds into $\Abelian{(\FinStepSolvQuo{H}{m-1})}$, whereas $(\FinStepSolvQuo{H}{m-1})^{[1]}$ has trivial image there.
	Therefore,
	\begin{equation*}
		\rho_H(g)\in \GeneSubgrpTop{x^N}\cap (\FinStepSolvQuo{H}{m-1})^{[1]}=1.
	\end{equation*}
	This proves
	\begin{equation*}
		\Centralizer{\FinStepSolvQuo{\mathcal{F}}{m}}{x^n}=\GeneSubgrpTop{x}
	\end{equation*}
	in the case where $r$ is finite.

	Finally, we consider the case $r=\infty$.
	Let $J$ be the directed set of finite subsets $X_j$ of $X$ such that $x\in X_j$.
	For each $j\in J$, let $\mathcal{F}_j$ be the finitely generated free pro-$\Sigma$ group on $X_j$ and let $\pi_j:\mathcal{F}\twoheadrightarrow \mathcal{F}_j$ be the continuous morphism sending generators in $X_j$ to themselves and generators in $X\setminus X_j$ to the identity element of $\mathcal{F}_{j}$.
	Additionally, let $\pi_{j}^{(m)}\colon\FinStepSolvQuo{\mathcal{F}}{m}\twoheadrightarrow \FinStepSolvQuo{\mathcal{F}_j}{m}$ be the natural projection induced from $\pi_{j}$.
	Then, by~\cite[Proposition 3.3.9]{MR2599132}, we have an isomorphism
	\begin{equation*}
		\FinStepSolvQuo{\mathcal{F}}{m}\, \xrightarrow{\sim}\, \varprojlim_{j\in J}\FinStepSolvQuo{\mathcal{F}_j}{m}.
	\end{equation*}
	Let $g\in \Centralizer{\FinStepSolvQuo{\mathcal{F}}{m}}{x^n}$.
	For each $j$, by the finite-rank case we obtain $\pi_j^{(m)}(g)\in \GeneSubgrpTop{\pi^{(m)}_j(x)}$.
	Passing to the inverse limit, we conclude that $g\in \GeneSubgrpTop{x}$.
	Thus the equality $\Centralizer{\FinStepSolvQuo{\mathcal{F}}{m}}{x^n}=\GeneSubgrpTop{x}$ also holds when $r=\infty$.
	This completes the proof.
\end{proof}

We say that a profinite group $G$ is \textit{slim} if the centralizer $\Centralizer{G}{H}$ of each open subgroup $H\subset G$ in $G$ is trivial (see~\cite[Definition 0.1]{MR2059759}).
We note that slimness implies center-freeness.

\begin{corollary}\label{cor_slim_free}
	Let $\mathcal{F}$ be a (possibly infinitely generated) free pro-$\Sigma$ group of rank $r$.
	Assume $r\geq 2$.
	Then, for any $m\in\mathbb{Z}_{\geq 2}$, the quotient $\FinStepSolvQuo{\mathcal{F}}{m}$ is slim.
\end{corollary}

\begin{proof}
	Let $X$ be a free generating set of $\mathcal{F}$.
	Let $H$ be an open subgroup of $\FinStepSolvQuo{\mathcal{F}}{m}$, and take two distinct elements $x,x'\in X$.
	Since $H$ is open, there exist $n,n'\geq 1$ such that $x^{n}\in H$ and $(x')^{n'}\in H$.
	Then Theorem~\ref{thm_center-free_freegroup} implies
	\begin{equation*}
		\Centralizer{\FinStepSolvQuo{\mathcal{F}}{m}}{H}
		\subset\Centralizer{\FinStepSolvQuo{\mathcal{F}}{m}}{x^n}\cap\Centralizer{\FinStepSolvQuo{\mathcal{F}}{m}}{(x')^{n'}}
		=\GeneSubgrpTop{x}\cap\GeneSubgrpTop{x'}
		=1,
	\end{equation*}
	where the last equality follows from the facts that $\GeneSubgrpTop{x}$ and $\GeneSubgrpTop{x'}$ embed into the abelianization $\Abelian{\mathcal{F}}$ and are distinct.
	This completes the proof.
\end{proof}

\section{The \texorpdfstring{$m$}{m}-step solvable Grothendieck conjecture}\label{section2}

In this section, we show that the maximal $m$-step solvable quotients of the geometric \'etale and tame fundamental groups of hyperbolic curves over a field are center-free (see Theorem~\ref{mainthm_center_free_fund}).
Moreover, we relate this result to the Grothendieck conjecture.
Throughout this section, let $\Sigma$ be a non-empty set of prime numbers.
For any profinite group $G$, we write $G^{\Sigma}$ for the maximal pro-$\Sigma$ quotient of $G$.

\subsection{Ab-torsion-freeness and ab-faithfulness}

\subsubsection{}
In this subsection, we introduce ab-torsion-freeness and ab-faithfulness for profinite groups, and record a strategy for proving center-freeness of maximal $m$-step solvable quotients.

\begin{definition}[{\cite[Definition~1.1]{Mochizuki2009CommentsStronglyTorsionFree}}]
	\label{defabfaithfulabtorsionfree}
	Let $G$ be a profinite group.
	\begin{enumerate}[(1)]
		\item\label{defabfaithfulabtorsionfree1}
		      We say that $G$ is \textit{ab-torsion-free} if, for each open subgroup $H$ of $G$, the abelianization $\Abelian{H}$ is torsion-free.
		\item\label{defabfaithfulabtorsionfree2}
		      We say that $G$ is \textit{ab-faithful} if, for each open subgroup $H$ of $G$ and each open normal subgroup $N$ of $H$, the natural morphism
		      \begin{equation*}
			      H/N \to \Aut\bigl(\Abelian{N}\bigr)
		      \end{equation*}
		      induced by conjugation is injective.
	\end{enumerate}
\end{definition}

\begin{remark}\label{rem:abfaithfulabtorsionfree}
	Let $G$ be a profinite group and let $m\in\mathbb{Z}_{\geq 2}$.
	For any open subgroup $P$ of $\FinStepSolvQuo{G}{m}$ that contains $\DerivSeries{(\FinStepSolvQuo{G}{m})}{m-1}$, let $\tilde{P}\subset G$ be its inverse image under $G\twoheadrightarrow \FinStepSolvQuo{G}{m}$.
	Then the natural morphism $\Abelian{\tilde{P}}\to \Abelian{P}$ is an isomorphism by Lemma~\ref{lemma:mstep}.
	In particular, the following hold:
	\begin{enumerate}[(i)]
		\item\label{rem:abfaithfulabtorsionfree3}
		      Assume that $G$ is ab-torsion-free, and let $H$ be an open subgroup of $\FinStepSolvQuo{G}{m}$.
		      If $\DerivSeries{(\FinStepSolvQuo{G}{m})}{m-1}\subset H$, then $\Abelian{H}$ is torsion-free.
		\item\label{rem:abfaithfulabtorsionfree4}
		      Assume that $G$ is ab-faithful.
		      Let $H$ be an open subgroup of $\FinStepSolvQuo{G}{m}$ and $N$ an open normal subgroup of $H$.
		      If $\DerivSeries{(\FinStepSolvQuo{G}{m})}{m-1}\subset N$, then the conjugation action of $H/N$ on $\Abelian{N}$ is also faithful.
	\end{enumerate}
	In what follows, we often only need these properties for open subgroups that contain $\DerivSeries{(\FinStepSolvQuo{G}{m})}{m-1}$.
\end{remark}

\begin{lemma}\label{abtorfreepropertieslemma}
	Let $G$ be an ab-torsion-free profinite group.
	\begin{enumerate}[(1)]
		\item\label{abtorfreepropertieslemma0}
		      Any closed subgroup of $G$ is ab-torsion-free.
		\item\label{abtorfreepropertieslemma1}
		      $G$ is torsion-free.
		\item\label{abtorfreepropertieslemma2}
		      For any $m\in\mathbb{Z}_{\geq 1}$, the quotient $\FinStepSolvQuo{G}{m}$ is torsion-free.
		\item\label{abtorfreepropertieslemma3}
		      Let $K$ be a closed normal subgroup of $G$ that is contained in $\DerivSeries{G}{1}$.
		      Then the conjugation action of $G/K$ on $\Abelian{K}$ has no nontrivial fixed points.
	\end{enumerate}
\end{lemma}

\begin{proof}
	\noindent\ref{abtorfreepropertieslemma0}
	Let $K$ be a closed subgroup of $G$.
	Since $G$ is profinite, we have
	\begin{equation*}
		K=\cap_{H} H,
	\end{equation*}
	where $H$ runs over all open subgroups of $G$ that contain $K$.
	Since projective limits commute with abelianization, we obtain $K^{\ab}\xrightarrow{\sim}\varprojlim_{H}H^{\ab}$.
	By the hypothesis, the right-hand side is torsion-free.
	Therefore, the group $K^{\ab}$ is also torsion-free.
	By applying the same argument with $K$ replaced by an arbitrary open subgroup of $K$, we conclude that $K$ is ab-torsion-free.

	\noindent\ref{abtorfreepropertieslemma1}
	Let $g\in G$ have finite order. Then the cyclic subgroup $\GeneSubgrp{g}$ is finite and hence closed.
	By~\ref{abtorfreepropertieslemma0}, we obtain $\Abelian{\GeneSubgrp{g}}=\GeneSubgrp{g}$ is torsion-free, hence $g=1$.
	Thus $G$ is torsion-free.

	\noindent\ref{abtorfreepropertieslemma2}
	We prove by induction on $m$.
	If $m=1$, then the assertion follows from the ab-torsion-freeness of $G$.
	Assume $m\ge 2$.
	Then we have the exact sequence:
	\begin{equation*}
		1
		\to \DerivSeries{(\FinStepSolvQuo{G}{m})}{m-1}
		\to \FinStepSolvQuo{G}{m}
		\to \FinStepSolvQuo{G}{m-1}
		\to 1
	\end{equation*}
	By the induction hypothesis, we have that $\FinStepSolvQuo{G}{m-1}$ is torsion-free.
	Moreover, by~\ref{abtorfreepropertieslemma0}, the quotient $\DerivSeries{(\FinStepSolvQuo{G}{m})}{m-1}=\Abelian{(\DerivSeries{G}{m-1})}$ is also torsion-free.
	Therefore $\FinStepSolvQuo{G}{m}$ is torsion-free.

	\noindent\ref{abtorfreepropertieslemma3}
	Let $N$ be an open normal subgroup of $G$.
	First, we claim that the natural morphism
	\begin{equation}\label{eq_hrthsrigoawe}
		(\Abelian{N})^{G/N} \to \Abelian{G}
	\end{equation}
	is injective.
	Indeed, consider the following natural morphism and transfer morphism:
	\begin{equation*}
		R_{N}\colon
		\Abelian{N}\to\Abelian{G},
		\qquad
		\mathrm{transfer}_{N}\colon
		\Abelian{G}\to \Abelian{N}
	\end{equation*}
	Let $G/N=\cup_{1\leq i\leq [G:N]} a_{i}N$ be a disjoint union of left cosets with representatives $\{a_{i}\}_{i}$.
	For each $n\in N$, we have $\mathrm{transfer}_{N}(R_{N}(n))=\sum (a_{i}^{-1}na_{i})$ on $\Abelian{N}$, i.e., we have
	\begin{equation*}
		\mathrm{transfer}_{N}\circ R_{N}=\sum_{a\in G/N} a\text{-conjugation}
	\end{equation*}
	on $\Abelian{N}$.
	In particular, the restricted morphism $(\mathrm{transfer}_{N}\circ R_{N})\mid_{(\Abelian{N})^{G/N}}$ coincides with multiplication by $[G:N]$.
	Since $G$ is ab-torsion-free, the group $\Abelian{N}$ is torsion-free.
	Hence $\mathrm{transfer}_{N}\circ R_{N}$ is injective on $(\Abelian{N})^{G/N}$.
	Therefore, the restricted morphism $(R_{N})\mid_{(\Abelian{N})^{G/N}}$, which is the morphism~\eqref{eq_hrthsrigoawe}, is also injective.
	This completes the proof of the claim.

	Let $\mathcal{N}$ be a set of all open normal subgroups that contain $K$.
	By running over all $N\in\mathcal{N}$, we have
	\begin{equation*}
		(\Abelian{K})^{G/K}
		=(\varprojlim_{N\in\mathcal{N}}(\Abelian{N}))^{G/K}
		=\varprojlim_{N\in\mathcal{N}}(\Abelian{N})^{G/K}
		=\varprojlim_{N\in\mathcal{N}}(\Abelian{N})^{G/N}
	\end{equation*}
	Therefore, this claim implies that the natural morphism
	\begin{equation*}
		(\Abelian{K})^{G/K}
		\to
		\Abelian{G}
	\end{equation*}
	is also injective.

	By taking the abelianization of the exact sequence $1\to K\to G\to G/K\to 1$, we have the exact sequence
	\begin{equation}\label{eq_grgwagawegw}
		(\Abelian{K})_{G/K}
		\to\Abelian{G}
		\to\Abelian{(G/K)}
		\to 1,
	\end{equation}
	where $(\Abelian{K})_{G/K}$ stands for the module of $G/K$-coinvariants of $\Abelian{K}$.
	Since $K\subset \DerivSeries{G}{1}$, the natural morphism $\Abelian{G}\xrightarrow{\sim}\Abelian{(G/K)}$ is an isomorphism.
	Hence the left-hand morphism of~\eqref{eq_grgwagawegw} is the zero map.
	The above claim implies that the composition of these natural morphisms
	\begin{equation*}
		(\Abelian{K})^{G/K}
		\hookrightarrow\Abelian{K}
		\twoheadrightarrow(\Abelian{K})_{G/K}
		\overset{0}\to
		\Abelian{G}.
	\end{equation*}
	is injective.
	Therefore, we obtain $(\Abelian{K})^{G/K}=1$.
	This completes the proof.
\end{proof}

\begin{lemma}\label{abfaithfulpropertieslemma}
	Let $G$ be an ab-faithful profinite group.
	\begin{enumerate}[(1)]
		\item\label{abfaithfulpropertieslemma1}
		      $G$ is center-free.
		\item\label{abfaithfulpropertieslemma2}
		      For any $m\in\mathbb{Z}_{\geq 1}$, we have $\CenterSubgrp{\FinStepSolvQuo{G}{m}}\subset\DerivSeries{(\FinStepSolvQuo{G}{m})}{m-1}$.
	\end{enumerate}
\end{lemma}

\begin{proof}
	\noindent\ref{abfaithfulpropertieslemma1}
	Let $N$ be an open normal subgroup of $G$.
	Then we have
	\begin{equation*}
		\CenterSubgrp{G}\subset \ker\Bigl(G/N \longrightarrow \Aut\bigl(\Abelian{N}\bigr)\Bigr).
	\end{equation*}
	By ab-faithfulness (applied to the pair $(H,N)=(G,N)$), the above morphism is injective,
	hence the kernel is trivial.
	Therefore, by running over all such $N$, we obtain
	\begin{equation*}
		\CenterSubgrp{G}\subset\bigcap_{N} N = \{1\}.
	\end{equation*}
	Thus the group $G$ is center-free.

	\noindent
	\ref{abfaithfulpropertieslemma2}
	The proof is essentially the same as the proof of~\ref{abfaithfulpropertieslemma1}.
	Let $N$ be an open normal subgroup of $\FinStepSolvQuo{G}{m}$ that contains $\DerivSeries{(\FinStepSolvQuo{G}{m})}{m-1}$.
	Then, by Remark~\ref{rem:abfaithfulabtorsionfree}\ref{rem:abfaithfulabtorsionfree4}, we have $\CenterSubgrp{\FinStepSolvQuo{G}{m}}\subset \ker\Bigl(\FinStepSolvQuo{G}{m}/N \longrightarrow \Aut\bigl(\Abelian{N}\bigr)\Bigr)$.
	By ab-faithfulness (applied to the pair $(H,N)=(\FinStepSolvQuo{G}{m},N)$), this morphism is injective,
	hence the kernel is trivial.
	Therefore, by running over all such $N$, we obtain
	\begin{equation*}
		\CenterSubgrp{\FinStepSolvQuo{G}{m}}\subset\bigcap_{N} N = \DerivSeries{(\FinStepSolvQuo{G}{m})}{m-1}.
	\end{equation*}
\end{proof}

The following proposition is the main result of this subsection.

\begin{proposition}\label{basic_prop_for_centerfree_m}
	Let $G$ be an ab-torsion-free ab-faithful profinite group.
	Then, for any $m\in\mathbb{Z}_{\geq 2}$, the quotient $\FinStepSolvQuo{G}{m}$ is torsion-free and center-free.
\end{proposition}

\begin{proof}
	The torsion-freeness follows from Lemma~\ref{abtorfreepropertieslemma}\ref{abtorfreepropertieslemma2}.
	We show the center-freeness.
	For any $a\in \FinStepSolvQuo{G}{m}$, the condition $a\in \CenterSubgrp{\FinStepSolvQuo{G}{m}}$ is equivalent to the condition that $gag^{-1}=a$ for every $g\in \FinStepSolvQuo{G}{m}$, and hence
	\begin{equation*}
		\CenterSubgrp{\FinStepSolvQuo{G}{m}}\cap \DerivSeries{(\FinStepSolvQuo{G}{m})}{m-1} = (\DerivSeries{(\FinStepSolvQuo{G}{m})}{m-1})^{\FinStepSolvQuo{G}{m-1}}.
	\end{equation*}
	By applying Lemma~\ref{abtorfreepropertieslemma}\ref{abtorfreepropertieslemma3} to $K=\DerivSeries{G}{m-1}$, the right-hand side is trivial.
	On the other hand, by Lemma~\ref{abfaithfulpropertieslemma}\ref{abfaithfulpropertieslemma2}, we have
	\begin{equation*}
		\CenterSubgrp{\FinStepSolvQuo{G}{m}}\subset\DerivSeries{(\FinStepSolvQuo{G}{m})}{m-1}.
	\end{equation*}
	Thus $\FinStepSolvQuo{G}{m}$ is center-free.
\end{proof}

\subsubsection{}
As we have already seen, a free pro-$\Sigma$ group is slim.
Since slimness is stronger than center-freeness, it is natural to ask whether it also holds for the maximal $m$-step solvable quotients of ab-faithful and ab-torsion-free profinite groups.
At the time of writing, the author does not know whether such groups are slim in general.
However, the following fact can be proved:

\begin{proposition}\label{climproposition}
	Let $G$ be an ab-torsion-free and ab-faithful profinite group.
	Then, for any $m\in\mathbb{Z}_{\geq 1}$ and any open subgroup $H$ of $\FinStepSolvQuo{G}{m}$, we have
	\begin{equation*}
		\Centralizer{\FinStepSolvQuo{G}{m}}{H}\subset \DerivSeries{(\FinStepSolvQuo{G}{m})}{m-1}.
	\end{equation*}
\end{proposition}

\begin{proof}
	Let $H\subset \FinStepSolvQuo{G}{m}$ be an open subgroup, and take
	$c\in \Centralizer{\FinStepSolvQuo{G}{m}}{H}$.
	Let $N\trianglelefteq_{\mathrm{open}} \FinStepSolvQuo{G}{m}$ be an open normal subgroup such that
	$\DerivSeries{(\FinStepSolvQuo{G}{m})}{m-1}\subset N$.
	Let $H_{N}\subset \Abelian{N}$ be the image of $H\cap N$ under the natural morphism
	$N\twoheadrightarrow \Abelian{N}$.
	Since $c$ centralizes $H$, conjugation by $c$ is trivial on $H\cap N$, hence it is also
	trivial on $H_{N}$.
	As $H\cap N$ is open in $N$, the subgroup $H_N$ is open (equivalently, of finite index) in $\Abelian{N}$.
	Therefore $\Abelian{N}/H_N$ is a torsion group, and hence the natural morphism
	\begin{equation*}
		H_{N}\otimes_{\mathbb{Z}}\mathbb{Q}\xrightarrow{\sim}\Abelian{N}\otimes_{\mathbb{Z}}\mathbb{Q}
	\end{equation*}
	is an isomorphism.
	It follows that conjugation by $c$ acts trivially on $\Abelian{N}\otimes_{\mathbb{Z}}\mathbb{Q}$.
	By Remark~\ref{rem:abfaithfulabtorsionfree}\ref{rem:abfaithfulabtorsionfree3},
	the group $\Abelian{N}$ is torsion-free.
	Thus the natural morphism $\Abelian{N}\hookrightarrow \Abelian{N}\otimes_{\mathbb{Z}}\mathbb{Q}$ is injective.
	Consequently, conjugation by $c$ is already trivial on $\Abelian{N}$.
	On the other hand, by Remark~\ref{rem:abfaithfulabtorsionfree}\ref{rem:abfaithfulabtorsionfree4},
	the conjugation action of $\FinStepSolvQuo{G}{m}/N$ on $\Abelian{N}$ is faithful, i.e., we have
	\begin{equation*}
		\ker\Bigl(\FinStepSolvQuo{G}{m} \to \Aut(\Abelian{N})\Bigr)=N.
	\end{equation*}
	Since $c$ acts trivially on $\Abelian{N}$, we obtain $c\in N$.
	By running over all such $N$, we conclude that
	\begin{equation*}
		c \in \bigcap_{N} N = \DerivSeries{(\FinStepSolvQuo{G}{m})}{m-1}.
	\end{equation*}
	This completes the proof.
\end{proof}

\subsection{Proof of the center-freeness of the maximal \texorpdfstring{$m$}{m}-step solvable quotients of the geometric fundamental groups of hyperbolic curves}

\subsubsection{}
In this subsection, we show that the maximal $m$-step solvable quotients of the geometric fundamental groups of smooth curves are center-free.
We always assume that smooth curves are geometrically connected.
For any smooth curve $X$ over a field $k$, we write
\begin{equation*}
	\EtFundGrpWithPt{X}{\ast}\qquad (\text{resp. }\EtFundGrpTameWithPt{X}{\ast})
\end{equation*}
for the \textit{\'etale fundamental group} (resp. \textit{tame fundamental group}) of $X$, where $\ast:\Spec(\Omega)\to X$ denotes a geometric point of $X$ and $\Omega$ denotes an algebraically closed field.
The fundamental group depends on the choice of base point only up to inner automorphisms, and therefore we omit the choice of base point below.

\subsubsection{}
We define the \textit{Euler characteristic} of $X$ by
\begin{equation}\label{eCrep}
	\chi(X)
	\coloneq
	\sum_{i=0}^{2}(-1)^{i}\dim_{\mathbb{Q}_{\ell}}(\mathrm{H}_{\et}^{i}(X,\mathbb{Q}_{\ell}))
	=
	\begin{cases}
		2-\dim_{\mathbb{Q}_{\ell}}(\mathrm{H}_{\et}^{1}(X,\mathbb{Q}_{\ell})) & (X:\ \text{proper}), \\
		1-\dim_{\mathbb{Q}_{\ell}}(\mathrm{H}_{\et}^{1}(X,\mathbb{Q}_{\ell})) & (X:\ \text{affine}).
	\end{cases}
\end{equation}
If $X$ is of type $(g,r)$, then a straightforward calculation shows that
\begin{equation*}
	\chi(X)=2-2g-r.
\end{equation*}
We say that $X$ is \textit{hyperbolic} if $\chi(X)<0$ (equivalently, if $(g,r)\notin\{(0,0), (0,1), (0,2), (1,0)\}$).
The basic fact about hyperbolicity is that, if $k$ is separably closed and $\ell$ is a prime number different from the characteristic of $k$, then
\begin{equation*}
	\EtFundGrp{X}^{\ell}\text{ is non-abelian}\text{ if and only if }\text{ $X$ is hyperbolic}
\end{equation*}
(see~\cite[Corollary~1.4]{MR1478817}).

\begin{lemma}\label{faithfuletlemma}
	Let $k$ be a separably closed field and let $\ell$ be a prime number different from the characteristic of $k$.
	Let $X$ be a hyperbolic curve over $k$, and let
	$f\colon Y \to X$ be a finite \'etale Galois covering with Galois group $\Gamma \coloneq \Gal(Y/X)$.
	Then the natural action
	\begin{equation*}
		\Gamma \curvearrowright \mathrm{H}_{\et}^{1}(Y,\mathbb{Q}_{\ell})
	\end{equation*}
	is faithful.
\end{lemma}

\begin{proof}
	Replacing $k$ by an algebraic closure does not change the statement.
	Hence we may assume that $k$ is algebraically closed.
	Let
	\begin{equation*}
		\Gamma_{0}\coloneq \ker\!\Bigl(
		\Gamma \longrightarrow \Aut_{\mathbb{Q}_{\ell}}\bigl(\mathrm{H}_{\et}^{1}(Y,\mathbb{Q}_{\ell})\bigr)
		\Bigr),
	\end{equation*}
	and put $Y_{0}\coloneq Y/\Gamma_{0}$.
	Then $Y\to Y_{0}$ is a finite \'etale Galois covering with Galois group $\Gamma_{0}$.
	Consider the Hochschild--Serre spectral sequence for the Galois covering $Y\to Y_{0}$ with coefficients $\mathbb{Q}_{\ell}$:
	\begin{equation*}
		E_{2}^{p,q}
		=
		\mathrm{H}^{p}\!\bigl(\Gamma_{0},\,\mathrm{H}_{\et}^{q}(Y,\mathbb{Q}_{\ell})\bigr)
		\ \Longrightarrow\
		\mathrm{H}_{\et}^{p+q}(Y_{0},\mathbb{Q}_{\ell}).
	\end{equation*}
	Then we obtain the associated five-term exact sequence
	\begin{equation*}
		0
		\to
		\mathrm{H}^{1}(\Gamma_{0},\mathbb{Q}_{\ell})
		\to
		\mathrm{H}_{\et}^{1}(Y_{0},\mathbb{Q}_{\ell})
		\to
		\mathrm{H}_{\et}^{1}(Y,\mathbb{Q}_{\ell})^{\Gamma_{0}}
		\to
		\mathrm{H}^{2}(\Gamma_{0},\mathbb{Q}_{\ell}).
	\end{equation*}
	Here $\Gamma_{0}$ is finite and hence $\mathrm{H}^{1}(\Gamma_{0},\mathbb{Q}_{\ell})=\mathrm{H}^{2}(\Gamma_{0},\mathbb{Q}_{\ell})=0$.
	Therefore, the restriction morphism
	\begin{equation*}
		\mathrm{H}_{\et}^{1}(Y_{0},\mathbb{Q}_{\ell})
		\xrightarrow{\ \sim\ }
		\mathrm{H}_{\et}^{1}(Y,\mathbb{Q}_{\ell})^{\Gamma_{0}}
	\end{equation*}
	is an isomorphism.
	By definition of $\Gamma_{0}$, the $\Gamma_{0}$-action on
	$\mathrm{H}_{\et}^{1}(Y,\mathbb{Q}_{\ell})$ is trivial, hence
	\begin{equation*}
		\dim_{\mathbb{Q}_{\ell}} (\mathrm{H}_{\et}^{1}(Y_{0},\mathbb{Q}_{\ell}))
		=
		\dim_{\mathbb{Q}_{\ell}} (\mathrm{H}_{\et}^{1}(Y,\mathbb{Q}_{\ell})).
	\end{equation*}
	Since the morphism $Y\to Y_{0}$ is finite, the curve $Y$ is proper if and only if $Y_{0}$ is proper.
	Hence (\ref{eCrep}) implies that $\chi(Y_{0})=\chi(Y)$.

	On the other hand, by the Riemann--Hurwitz theorem for their compactifications of $Y\to Y_{0}$ (see~\cite[Corollary~2.4]{MR0463157}), we have the inequality
	\begin{equation*}
		\chi(Y) \leq d\cdot\chi(Y_{0}),
	\end{equation*}
	where $d\coloneq \#\Gamma_{0}$.
	By the hypothesis, the curve $X$ is hyperbolic, and hence $\chi(Y)<0$.
	This implies that $d=1$, i.e., the action $\Gamma\curvearrowright \mathrm{H}_{\et}^{1}(Y,\mathbb{Q}_{\ell})$ is faithful.
\end{proof}

\begin{proposition}\label{prop_fund_abs}
	Let $X$ be a hyperbolic curve over a field $k$.
	Assume that $k$ is a separably closed field and that $\Sigma$ contains a prime number different from the characteristic of $k$.
	Then the groups $\EtFundGrp{X}^{\Sigma}$ and $\EtFundGrpTame{X}^{\Sigma}$ are both ab-torsion-free and ab-faithful.
\end{proposition}

\begin{proof}
	For simplicity, we write
	\begin{equation*}
		\EtFundGrpGeom{X}\coloneq \EtFundGrp{X}^{\Sigma}\qquad(\text{resp. }\EtFundGrpTame{X}^{\Sigma}).
	\end{equation*}
	Since any open subgroup of $\Delta$ is also an \'etale fundamental group of a hyperbolic curve over $k$, the known result~\cite[Corollary~1.2]{MR1478817} implies that $\EtFundGrpGeom{X}$ is ab-torsion-free.

	Next, we show the ab-faithfulness.
	Let $H$ be an open subgroup of $\EtFundGrpGeom{X}$ and $N$ an open normal subgroup of $H$.
	To prove ab-faithfulness, we may replace $\EtFundGrpGeom{X}$ by $H$ and assume that $H=\EtFundGrpGeom{X}$.
	Let $Y\to X$ be the connected finite \'etale Galois covering corresponding to $N$, with Galois group $\Gamma\coloneq\EtFundGrpGeom{X}/N$.
	Let $\ell\in\Sigma$ be a prime number different from the characteristic of $k$.
	Then the action
	\begin{equation*}
		\Gamma\curvearrowright \mathrm{H}_{\et}^{1}(Y,\mathbb{Q}_{\ell})
	\end{equation*}
	is faithful by Lemma~\ref{faithfuletlemma}.
	On the other hand, the $\Gamma$-module $\mathrm{H}_{\et}^{1}(Y,\mathbb{Q}_{\ell})$ is the $\mathbb{Q}_{\ell}$-linear dual of $N^{\ab,\ell}\otimes_{\mathbb{Z}_{\ell}}\mathbb{Q}_{\ell}$, with the conjugation action of $\Gamma$ (see~\cite[Expos\'e XI, Section 5]{MR0354651}).
	Therefore, the composition of the natural morphisms
	\begin{equation*}
		\Gamma=\EtFundGrpGeom{X}/N\to \Aut(\Abelian{N})\to \Aut_{\mathbb{Z}_{\ell}}(N^{\ab,\ell})\to \Aut_{\mathbb{Q}_{\ell}}\bigl(N^{\ab,\ell}\otimes_{\mathbb{Z}_{\ell}}\mathbb{Q}_{\ell}\bigr).
	\end{equation*}
	is injective.
	This proves that $\EtFundGrpGeom{X}$ is ab-faithful.
	This completes the proof.
\end{proof}

\subsubsection{}
The following is the first main theorem of this paper:

\begin{theorem}\label{mainthm_center_free_fund}
	Let $X$ be a hyperbolic curve over a field $k$.
	Assume that $k$ is a separably closed field and that $\Sigma$ contains a prime number different from the characteristic of $k$.
	Then, for any $m\in\mathbb{Z}_{\geq 2}$, the maximal $m$-step solvable quotients of $\EtFundGrp{X}^{\Sigma}$ and $\EtFundGrpTame{X}^{\Sigma}$ are both torsion-free and center-free.
\end{theorem}

\begin{proof}
	The assertion follows from Propositions~\ref{basic_prop_for_centerfree_m} and~\ref{prop_fund_abs}.
\end{proof}

\begin{corollary}\label{center_free_surface}
	For any $m\in\mathbb{Z}_{\geq 2}$, the maximal $m$-step solvable quotients of a pro-$\Sigma$ surface group of genus $g\geq 2$ are torsion-free and center-free.
\end{corollary}

\begin{proof}
	There exists a smooth proper curve over an algebraically closed field whose pro-$\Sigma$ \'etale fundamental group is isomorphic to the pro-$\Sigma$ surface group.
	Thus, the assertion follows from Theorem~\ref{mainthm_center_free_fund}.
\end{proof}

\subsection{Injectivity of the \texorpdfstring{$m$}{m}-step solvable Grothendieck conjecture}

\subsubsection{}
Next, we explain an application of Theorem~\ref{mainthm_center_free_fund} to the $m$-step solvable analogue of the Grothendieck conjecture.
Let $X$ be a smooth curve over a field $k$.
In this subsection, we focus only on the case where the field $k$ has characteristic $0$ (or, more restrictively, the field $k$ is a sub-$p$-adic field for some prime number $p$, i.e., a field that embeds as a subfield of a finitely generated extension of $\mathbb{Q}_{p}$).
For simplicity, we set
\begin{equation*}
	\EtFundGrpGeom{X}
	\coloneq\EtFundGrp{X_{\overline{k}}}^{\Sigma},
	\qquad\text{and}\qquad
	\EtFundGrpPi{X}
	\coloneq\EtFundGrp{X}/\ker (\EtFundGrp{X_{\overline{k}}}\to\EtFundGrpGeom{X}),
\end{equation*}
where $\overline{k}$ is an algebraic closure of $k$.
In this notation, we have the following exact sequence, called the \textit{homotopy exact sequence}:
\begin{equation*}
	1
	\to\EtFundGrpGeom{X}
	\to\EtFundGrpPi{X}
	\to\AbsGalGrp{k}
	\to 1.
\end{equation*}
We also define
\begin{equation*}
	\FinStepSolvQuoGeom{\EtFundGrpPi{X}}{m}
	\coloneq
	\EtFundGrpPi{X}/\DerivSeries{\EtFundGrpGeom{X}}{m}.
\end{equation*}
By construction, the homotopy exact sequence naturally induces the following exact sequence:
\begin{equation}\label{hom_m}
	1
	\to\FinStepSolvQuo{\EtFundGrpGeom{X}}{m}
	\to\FinStepSolvQuoGeom{\EtFundGrpPi{X}}{m}
	\to\AbsGalGrp{k}
	\to 1.
\end{equation}

\subsubsection{}
Let $i$ range over $\{1,2\}$.
Let $m\in\mathbb{Z}_{\geq 1}$.
Let $X_{i}$ be a smooth curve over $k$.
We write $\UnivCov{X_{i}}^{m}\to X_{i}$ for the maximal geometrically $m$-step solvable pro-$\Sigma$ Galois covering of $X_{i}$, which is a scheme over $\overline{k}$.
For this, we introduce the following non-standard notation for isomorphism sets:
\medskip
\begin{itemize}
	\item
	      We denote by
	      \begin{equation*}
		      \Isom_{\overline{k}/k}\bigl(\UnivCov{X_{1}}^{m}/X_{1}, \UnivCov{X_{2}}^{m}/X_{2}\bigr)
	      \end{equation*}
	      the set of all pairs
	      \begin{equation*}
		      \left\{(\tilde{\mathrm{pr}_{G}},\mathrm{pr}_{G})\in
		      \Isom_{\overline{k}}(\UnivCov{X_{1}}^{m}, \UnivCov{X_{2}}^{m})\times
		      \Isom_{k}(X_{1}, X_{2})
		      \,\middle|\,
		      \vcenter{
		      \xymatrix{
		      \UnivCov{X_{1}}^{m} \ar[r]^{\tilde{\mathrm{pr}_{G}}} \ar[d] & \UnivCov{X_{2}}^{m} \ar[d] \\
		      X_{1} \ar[r]^{\mathrm{pr}_{G}} & X_{2}
		      }}
		      \text{ commutes.}
		      \right\}.
	      \end{equation*}
	\item
	      Let $n\in\mathbb{Z}_{\geq 0}$.
	      We denote by
	      \begin{equation*}
		      \Isom_{\AbsGalGrp{k}}^{(m+n)}(\FinStepSolvQuoGeom{\EtFundGrpPi{X_{1}}}{m}, \FinStepSolvQuoGeom{\EtFundGrpPi{X_{2}}}{m})
	      \end{equation*}
	      the image of the natural map
	      \begin{equation*}
		      \Isom_{\AbsGalGrp{k}}\bigl(\FinStepSolvQuoGeom{\EtFundGrpPi{X_{1}}}{m+n}, \FinStepSolvQuoGeom{\EtFundGrpPi{X_{2}}}{m+n}\bigr)
		      \to
		      \Isom_{\AbsGalGrp{k}}\bigl(\FinStepSolvQuoGeom{\EtFundGrpPi{X_{1}}}{m}, \FinStepSolvQuoGeom{\EtFundGrpPi{X_{2}}}{m}\bigr).
	      \end{equation*}
	      We also define
	      \begin{equation*}
		      \Isom_{\AbsGalGrp{k}}^{\Out,(m+n)}(\FinStepSolvQuoGeom{\EtFundGrpPi{X_{1}}}{m}, \FinStepSolvQuoGeom{\EtFundGrpPi{X_{2}}}{m})\coloneq \Isom_{\AbsGalGrp{k}}^{(m+n)}(\FinStepSolvQuoGeom{\EtFundGrpPi{X_{1}}}{m}, \FinStepSolvQuoGeom{\EtFundGrpPi{X_{2}}}{m})/\Inn(\FinStepSolvQuo{\EtFundGrpGeom{X_{2}}}{m}),
	      \end{equation*}
	      where $\Inn(\FinStepSolvQuo{\EtFundGrpGeom{X_{2}}}{m})$ denotes the subgroup of $\Aut_{\AbsGalGrp{k}}(\FinStepSolvQuoGeom{\EtFundGrpPi{X_{2}}}{m})$ consisting of inner automorphisms induced by conjugation by elements of $\FinStepSolvQuo{\EtFundGrpGeom{X_{2}}}{m}$.
\end{itemize}
\medskip
With the above notation, S.~Mochizuki proved the following result, which is called the \textit{$m$-step solvable Grothendieck conjecture for hyperbolic curves}:

\begin{theorem*}[{\cite[Theorem~18.1]{MR1720187}}]
	Assume $\Sigma=\{p\}$.
	Let $i$ range over $\{1,2\}$.
	Let $m\in\mathbb{Z}_{\geq 2}$.
	Let $k$ be a sub-$p$-adic field with algebraic closure $\overline{k}$, and let $X_{i}$ be a smooth curve over $k$.
	Assume that at least one of $X_{1}$ and $X_{2}$ is hyperbolic.
	Then the natural map
	\begin{equation}\label{m-step-morphism-mochi}
		\Isom_{\overline{k}/k}\bigl(\UnivCov{X_{1}}^{m}/X_{1}, \UnivCov{X_{2}}^{m}/X_{2}\bigr)
		\to
		\Isom_{\AbsGalGrp{k}}^{(m+3)}\bigl(\FinStepSolvQuoGeom{\EtFundGrpPi{X_{1}}}{m}, \FinStepSolvQuoGeom{\EtFundGrpPi{X_{2}}}{m}\bigr)
	\end{equation}
	is surjective.
\end{theorem*}

\subsubsection{}
The following is the second main theorem of this paper:

\begin{theorem}\label{cor:relative_anabelian}
	We keep the notation and assumptions as in the above theorem.
	Then the natural map~\eqref{m-step-morphism-mochi} is bijective.
\end{theorem}

\begin{proof}
	If $\Isom_{\AbsGalGrp{k}}^{(m+3)}\bigl(\FinStepSolvQuoGeom{\EtFundGrpPi{X_{1}}}{m}, \FinStepSolvQuoGeom{\EtFundGrpPi{X_{2}}}{m}\bigr)=\emptyset$, then the statement is tautological.
	Hence we may assume that $\Isom_{\AbsGalGrp{k}}^{(m+3)}\bigl(\FinStepSolvQuoGeom{\EtFundGrpPi{X_{1}}}{m}, \FinStepSolvQuoGeom{\EtFundGrpPi{X_{2}}}{m}\bigr)\neq \emptyset$.
	First, by Theorem~\ref{mainthm_center_free_fund}, the group $\FinStepSolvQuo{\EtFundGrpGeom{X_{1}}}{m}$ is nontrivial and center-free if $X_{1}$ is hyperbolic.
	If $X_{1}$ is not hyperbolic, then $\FinStepSolvQuo{\EtFundGrpGeom{X_{1}}}{m}$ is abelian.
	Therefore, we can determine whether $X_{1}$ is hyperbolic from $\FinStepSolvQuo{\EtFundGrpGeom{X_{1}}}{m}$.
	Hence we may assume that $X_{1}$ and $X_{2}$ are both hyperbolic.
	Next, by definition, there is an exact sequence:
	\begin{equation*}
		1\to \Inn(\FinStepSolvQuo{\EtFundGrpGeom{X_{2}}}{m})\to
		\Isom_{\AbsGalGrp{k}}^{(m+3)}\bigl(\FinStepSolvQuoGeom{\EtFundGrpPi{X_{1}}}{m}, \FinStepSolvQuoGeom{\EtFundGrpPi{X_{2}}}{m}\bigr)\to
		\Isom_{\AbsGalGrp{k}}^{\Out,(m+3)}\bigl(\FinStepSolvQuoGeom{\EtFundGrpPi{X_{1}}}{m}, \FinStepSolvQuoGeom{\EtFundGrpPi{X_{2}}}{m}\bigr)\to 1.
	\end{equation*}
	On the geometric side, we have an exact sequence:
	\begin{equation*}
		1\to
		\Aut_{X_{2,\overline{k}}}\bigl(\UnivCov{X_{2}}^{m}\bigr)\to
		\Isom_{\overline{k}/k}\bigl(\UnivCov{X_{1}}^{m}/X_{1}, \UnivCov{X_{2}}^{m}/X_{2}\bigr)\to
		\Isom_k(X_{1}, X_{2})\to 1.
	\end{equation*}
	Therefore, we obtain a commutative diagram with exact rows:
	\begin{equation*}
		\xymatrix@C=16pt{
		1\ar[r]&
		\Aut_{X_{2,\overline{k}}}\bigl(\UnivCov{X_{2}}^{m}\bigr)\ar@{->>}[d]\ar[r]&
		\Isom_{\overline{k}/k}\bigl(\UnivCov{X_{1}}^{m}/X_{1}, \UnivCov{X_{2}}^{m}/X_{2}\bigr)\ar[d]\ar[r]&
		\Isom_{k}(X_{1}, X_{2})\ar[d]\ar[r]&1\\
		1\ar[r]&
		\Inn(\FinStepSolvQuo{\EtFundGrpGeom{X_{2}}}{m})\ar[r]&
		\Isom_{\AbsGalGrp{k}}^{(m+3)}\bigl(\FinStepSolvQuoGeom{\EtFundGrpPi{X_{1}}}{m}, \FinStepSolvQuoGeom{\EtFundGrpPi{X_{2}}}{m}\bigr)\ar[r]&
		\Isom_{\AbsGalGrp{k}}^{\Out,(m+3)}\bigl(\FinStepSolvQuoGeom{\EtFundGrpPi{X_{1}}}{m}, \FinStepSolvQuoGeom{\EtFundGrpPi{X_{2}}}{m}\bigr)\ar[r]&1.
		}
	\end{equation*}
	By the definition of $\UnivCov{X_{2}}^{m}\to X_{2}$, we have a canonical identification
	\begin{equation*}
		\Aut_{X_{2,\overline{k}}}\bigl(\UnivCov{X_{2}}^{m}\bigr)\, \cong\, \FinStepSolvQuo{\EtFundGrpGeom{X_{2}}}{m}.
	\end{equation*}
	Therefore, the group 
	\begin{equation*}
		\ker(\Aut_{X_{2,\overline{k}}}\bigl(\UnivCov{X_{2}}^{m}\bigr)\twoheadrightarrow \Inn(\FinStepSolvQuo{\EtFundGrpGeom{X_{2}}}{m}))=\CenterSubgrp{\FinStepSolvQuo{\EtFundGrpGeom{X_{2}}}{m}}
	\end{equation*}
	is trivial, since the group $\FinStepSolvQuo{\EtFundGrpGeom{X_{2}}}{m}$ is center-free by Theorem~\ref{mainthm_center_free_fund}.
	Hence the left-hand vertical arrow in the above commutative diagram is bijective.
	Moreover, the right-hand vertical arrow is surjective by~\cite[Theorem~18.1]{MR1720187}, and injective by~\cite[Lemma~4.9]{MR4745885}.
	(Note that~\cite[Lemma~4.9]{MR4745885} assumed that $k$ is a field finitely generated over $\mathbb{Q}$.
	However, the proof can be applied to the case where $k$ is a sub-$p$-adic field.)
	Thus, by the diagram chase, the middle vertical arrow is also bijective.
	This completes the proof.
\end{proof}

\begin{remark}
	In Theorem~\ref{cor:relative_anabelian}, we assumed that $\Sigma=\{p\}$.
	If we further assume that $m\geq 3$, this assumption can be weakened to $p\in\Sigma$.
	The author expects that the same statement should hold for such $\Sigma$ even when $m=2$.
	However, to prove this, we would need to check whether the proof of~\cite[Theorem~18.1]{MR1720187} applies in this setting as well.
	At the time of writing, the author has not attempted this modification.
\end{remark}

\section*{Acknowledgements}
The author would like to express sincere gratitude to Prof.~Akio~Tamagawa for his invaluable assistance and insightful suggestions throughout this research.

\printbibliography

\end{document}